\documentclass[12pt]{amsart}
\usepackage[latin2]{inputenc}
\usepackage{amsmath,amssymb,amsbsy,amsfonts}
\usepackage{color}
\usepackage{graphics}
\usepackage{epsfig,psfrag,graphicx}
\usepackage{amssymb}
\usepackage{eepic,epic}

\usepackage{epsfig} 
\usepackage{amsmath} 
\usepackage{amssymb} 
\usepackage{amsbsy} 

\title{Transport equation with integral terms}
\author[C.~De Lellis]{Camillo De Lellis}
\address{Camillo De Lellis, Institut f\"ur Mathematik, Universit\"at Z\"urich,
Winterthurerstrasse~190, CH-8057 Z\"urich, Switzerland}
\email{delellis@math.unizh.ch}

\author[P.~Gwiazda]{Piotr Gwiazda}
\address{Piotr Gwiazda, University of Warsaw, Institute of Applied Mathematics and Mechanics,
Banacha~2, 02-097 Warsaw, Poland}
\email{pgwiazda@mimuw.edu.pl}

\author[A.~\'{S}wierczewska-Gwiazda]{Agnieszka \'{S}wierczewska-Gwiazda}
\address{Agnieszka \'{S}wierczewska-Gwiazda, University of Warsaw, Institute of Applied Mathematics and Mechanics,
Banacha~2, 02-097 Warsaw, Poland}
\email{aswiercz@mimuw.edu.pl}

\thanks{ A.~\'{S}wierczewska-Gwiazda was supported  by the project IdP2011/000661.  Piotr Gwiazda is a coordinator of International Ph.D.
Projects Programme of Foundation for Polish Science operated within the Innovative Economy Operational
Programme 2007-2013 (Ph.D. Programme: Mathematical Methods in Natural Sciences).}

\begin{document}
\maketitle \maketitle
\newtheorem{definition}{Definition}[section]
\newtheorem{theorem}{Theorem}[section]
\newtheorem{lemma}[theorem]{Lemma}
\newtheorem{corollary}[theorem]{Corollary}
\newtheorem{proposition}[theorem]{Proposition}
\newtheorem{example}{Example}
\newtheorem{remark}[theorem]{Remark}
\newtheorem{ipotesi}[theorem]{Assumption}
\renewcommand{\theequation}{\thesection.\arabic{equation}}
\newcommand{\nc}{\newcommand}
\nc{\V}{{\cal V}} 
\nc{\M}{{\cal M}} 
\nc{\T}{{{\mathbb R}^n}}  
\nc{\D}{{\cal D}} 
\nc{\R}{{\mathbb R}} 
\nc{\N}{{\mathbb N}}
\nc{\weak}{\rightharpoonup}
\nc{\weakstar}{\stackrel{\ast}{\rightharpoonup}} 
\def\bbbone{{\mathchoice {\rm 1\mskip-4mu l}
{\rm 1\mskip-4mu l} {\rm 1\mskip-4.5mu l} {\rm 1\mskip-5mu l}}}
\renewcommand{\div}{{\mathrm{div}_x}\,}
\def\Rdp{\mathbb{R}_+\times \mathbb{R}^{n}}
\def\bG{A}
\def\prob{\mathop{\mathrm{Prob}}\nolimits}
\def\sgn{\mathop{\mathrm{sgn}}\nolimits}

\def\note#1{\marginpar{\small #1}}

\def\tens#1{\pmb{\mathsf{#1}}}
\def\vec#1{\boldsymbol{#1}}

\def\norm#1{\left|\!\left| #1 \right|\!\right|}
\def\fnorm#1{|\!| #1 |\!|}
\def\abs#1{\left| #1 \right|}
\def\ti{\text{I}}
\def\tii{\text{I\!I}}
\def\tiii{\text{I\!I\!I}}
\def\diam{\mathop{\mathrm{diam}}\nolimits}

\def\diver{\mathop{\mathrm{div}}\nolimits}
\def\grad{\mathop{\mathrm{grad}}\nolimits}
\def\Div{\mathop{\mathrm{Div}}\nolimits}
\def\Grad{\mathop{\mathrm{Grad}}\nolimits}

\def\tr{\mathop{\mathrm{tr}}\nolimits}
\def\cof{\mathop{\mathrm{cof}}\nolimits}
\def\det{\mathop{\mathrm{det}}\nolimits}

\def\lin{\mathop{\mathrm{span}}\nolimits}
\def\pr{\noindent \textbf{Proof: }}
\def\pp#1#2{\frac{\partial #1}{\partial #2}}
\def\dd#1#2{\frac{\d #1}{\d #2}}

\def\RR{\mathcal{R}}
\def\Re{\mathbb{R}}
\def\bx{\vec{x}}
\def\be{\vec{e}}
\def\bef{\vec{f}}
\def\bec{\vec{c}}
\def\bs{\vec{s}}
\def\ba{\vec{a}}
\def\bn{\vec{n}}
\def\bphi{\vec{\varphi}}
\def\btau{\vec{\tau}}
\def\bc{\vec{c}}
\def\bz{\vec{z}}
\def\bg{\vec{g}}
\def\measure{\mathop{\mathrm{meas}}\nolimits}
\def\mM{\mathcal{M}}
\def\mL#1#2{\mathcal{L}^{\alpha_{#2}}_{#1}}

\def\bH{\tens{H}}
\def\bG{\tens{G}}
\def\bW{\tens{W}}
\def\bA{\tens{A}}
\def\bT{\tens{T}}
\def\bD{\tens{D}}
\def\bF{\tens{F}}
\def\bB{\tens{B}}
\def\bV{\tens{V}}
\def\bS{\tens{S}}
\def\bI{\tens{I}}
\def\bi{\vec{i}}
\def\bv{\vec{v}}
\def\bfi{\vec{\varphi}}
\def\bk{\vec{k}}
\def\b0{\vec{0}}
\def\bom{\vec{\omega}}
\def\bw{\vec{w}}
\def\p{\pi}
\def\bu{\vec{u}}
\def\bq{\vec{q}}

\def\obS{\bS}
\def\obD{\overline{\bD}}

\def\ID{\mathcal{I}_{\bD}}
\def\IP{\mathcal{I}_{p}}
\def\Pn{(\mathcal{P})}
\def\Pe{(\mathcal{P}^{\frac{1}{n} })}
\def\Pee{(\mathcal{P}^{\eta , \frac{1}{n} })}

\def\Ln#1{L^{#1}_{\bn}}

\def\Wn#1{W^{1,#1}_{\bn}}

\def\Lnd#1{L^{#1}_{\bn, \diver}}

\def\Wnd#1{W^{1,#1}_{\bn, \diver}}

\def\Wndm#1{W^{-1,#1}_{\bn, \diver}}

\def\Wnm#1{W^{-1,#1}_{\bn}}

\def\Lb#1{L^{#1}(\partial \Omega)}

\def\Lnt#1{L^{#1}_{\bn, \btau}}

\def\Wnt#1{W^{1,#1}_{\bn, \btau}}

\def\Wntd#1{W^{1,#1}_{\bn, \btau, \diver}}

\def\Wntdm#1{W^{-1,#1}_{\bn,\btau, \diver}}

\def\Wntm#1{W^{-1,#1}_{\bn, \btau}}

\renewcommand{\div}{{\rm div}}
\def\bff{\vec{f}}
\def\bxi{\vec{\xi}}
\def\bzeta{\vec{\zeta}}
\def\bG{\tens{G}}
\def\bB{\tens{B}}
\def\balpha{\tens{\alpha}}
\newcommand{\Ref}[1]{{\rm(\ref{#1})}}
\def\bDv{\tens{D}\vec{v}}
\newcommand{\Rd}{{{\mathbb R}^{d\times d}_{\rm sym}}}
\newcommand{\Rtri}{{{\mathbb R}^{3\times 3}_{\rm sym}}}
\def\bd{\tens{d}}
\newcommand{\mA}{\mathcal{A}}
\newcommand{\mLl}{\mathcal{L}}
\newcommand{\mB}{\mathcal{B}}
\newcommand{\vrho}{\varrho}
\newcommand{\vep}{\varepsilon}
\def\bbbone{{\mathchoice {\rm 1\mskip-4mu l}
{\rm 1\mskip-4mu l} {\rm 1\mskip-4.5mu l} {\rm 1\mskip-5mu l}}}
\newcommand{\modular}[1]{{\stackrel{ #1}{\longrightarrow\,}}}
\newcommand{\vd}{\bar{v}} \newcommand{\zd}{\bar{z}}

\def\Xint#1{\mathchoice
{\XXint\displaystyle\textstyle{#1}}%
{\XXint\textstyle\scriptstyle{#1}}%
{\XXint\scriptstyle\scriptscriptstyle{#1}}%
{\XXint\scriptscriptstyle\scriptscriptstyle{#1}}%
\!\int}
\def\XXint#1#2#3{{\setbox0=\hbox{$#1{#2#3}{\int}$ }
\vcenter{\hbox{$#2#3$ }}\kern-.6\wd0}}
\def\ddashint{\Xint=}
\def\dashint{\Xint-}

\def\dd{{\rm d}}

\begin{abstract}
We prove some theorems on the existence, uniqueness, stability and compactness properties of solutions to inhomogeneous transport equations with Sobolev coefficients, where the inhomogeneous term depends upon the solution through an integral operator. Contrary to the usual DiPerna-Lions approach, the essential step is to formulate the problem in the Lagrangian setting. Some motivations to study the above problem arise from the description of polymeric flows, where such kind of equations are coupled with other Navier-Stokes type equations. 
Using the results for the transport equation we will provide, in a separate paper, a sequential stability theorem for the full problem of the flow of concentrated polymers. 
 \end{abstract}

\section{Introduction}
Consider the Cauchy problem
\begin{align}\label{E1}
{\partial}_t  u(t,x,r)+ b(t,x,r)\nabla_{x,r} u(t,x,r) =f(t,x,r)[u]&\quad\mbox{ in }[0,T]\times\T\times\R^j,\\ \label{E0}
u(0,\cdot)=u_0 &\quad\mbox{ in }  \T\times\R^j.
\end{align}
where 
\begin{equation}\label{f}
f(t,x,r)[u]=\int_{\R^j}\gamma(t,x,r,\tilde r)u(t,x,\tilde r)\;d\tilde r, 
\end{equation}
and $b:[0,T]\times\T\times\R^j\to\R^{n+j}$. The kernel 
$\gamma:[0,T]\times\R^n\times\R^j\times\R^j\to\R$ will be specified later. The notation $(x,r)$ for the space variables is used to underline that the integration might be taken with respect to only part of the space variables, namely only with respect to $r$. This has important consequences and increases the difficulty in treating this term, an issue which will be discussed further.

Our motivation to study \eqref{E1}-\eqref{f}, which we describe in more detail in the sequel, arises from a model for polymeric flows, where the transport equation describes the evolution of a suitable microscopic quantity. The first component (which we will call $b^{(1)}$) of the transport coefficient $b$ is then the velocity of a solvent which satisfies some Navier-Stokes type equations. For this reason it is natural to expect a Sobolev regularity for $b$. 
In the rest of the note we will indeed make the following assumptions on the vector field~$b$:

\begin{ipotesi}\label{main_A}
$\;$
\begin{enumerate}
\item[(B1)] 
$b(t,x,r)=(b^{(1)}(t,x),b^{(2)}(t,x,r))$, where  $\div_xb^{(1)}\in L^1(0,T; L^\infty(\R^{n}\times\R^j))$  \\and  $\div_rb^{(2)}\in L^1(0,T; L^\infty(\R^{n}\times\R^j)),$
\item[(B2)]
$b\in L^1(0,T;W^{1,q}_{loc}(\R^{n}\times\R^j))^{n+j}$ for some $q\geq 1$
\item[(B3)] 
\begin{equation}\label{b1b2}
\frac{b(t,x,r)}{1+|x|+|r|}\in L^1(0,T; L^1(\R^n\times\R^j))^{n+j}+L^1(0,T; L^\infty(\R^n\times\R^j))^{n+j}\, .
\end{equation}
In what follows it will be sometimes useful to refer to Assumption \eqref{b1b2} in the form of $b = b_1 + b_2$, where
$\frac{b_1}{1+|x|+|r|}\in L^1 (0,T;L^1)^{n+j}$ and $\frac{b_2}{1+|x|+|r|} \in L^1 (0,T;L^\infty)^{n+j}$. 
\end{enumerate}
\end{ipotesi}

The following theorem summarizes the conclusions of this note about the solutions of \eqref{E1}-\eqref{E0}. 
 
\begin{theorem}[Existence, Uniqueness and Stability] Assume that $b\in L^1 (0,T; L^{p'}_{loc}(\R^n\times\R^j))^{n+j}$ satisfies Assumption \ref{main_A}, $\gamma\in  L^1(0,T; L^\infty (\R^n;L^p(\R^j;L^{p'}(\R^j)))$
and $u_0\in L^p(\R^n\times\R^j)$, where $\frac{1}{p} + \frac{1}{p'}=1$ and $1\leq p < \infty$. Then there exists a unique solution  $u\in L^\infty(0,T;L^p(\R^n\times\R^j))$ to \eqref{E1}-\eqref{f}. 

Moreover, let $b_k\in L^1 (0,T; L^{p'}_{loc}(\R^n\times\R^j))^{n+j}$ satisfy Assumption \ref{main_A} (with uniform bounds in the corresponding conditions), where $b_k, \div\, b_k$   converge to $b \in L^1(0,T;L^{p'}_{loc}(\T\times\R^j))^{n+j}$ and $ \div\, b\in L^1(0,T;L^{p'}_{loc}(\T\times\R^j)),$ as $k\to\infty$ respectively. Let $u_k$ be a solution to 
\eqref{E1}-\eqref{f} with $b_k$ in place of $b$, but with the same initial condition $u_0$. Then $u_k$ converges in ${\mathcal C}([0,T];L^p_{loc}(\T\times\R^j))$ to a solution $u$ of \eqref{E1}-\eqref{f}. 
\label{Th:stabilityI}
\end{theorem}

\begin{remark}
Notice that the assumption $b\in L^1 (0,T; L^{p'}_{loc})$ is needed to make sense of the product $u b$ as an $L^1_{loc}$ function. 
However, following the Lagrangian formulation of the problem used in our proof, we can make
sense of solutions of \eqref{E1}-\eqref{f} even without such assumption, see Remark \ref{r:makes_sense} below.
\end{remark}

\subsection{Lagrangian versus Eulerian approach}
The problem of transport equations with Sobolev coefficients was addressed in the famous seminal paper~\cite{DiLi1989} of DiPerna and Lions, where the authors introduced their powerful theory of renormalized solutions.  For simplicity we recall the theory in the autonomous case.
Consider the problem
\begin{equation}\label{e:example}
u_t(t,y)+b(y)\nabla_yu(t,y)=0, \quad
u(0,\cdot)=u_0\, .
\end{equation}
If $b$ has only limited regularity, we regard a weak solution of the above problem as solving $u_t + \div\, ( ub) - u \div\, b = 0$, where we assume to
have enough summability to justify all the products involved (in particular the distributional divergence must be at least $L^1_{loc}$). According to DiPerna and Lions a solution $u$ is renormalized if it satisfies
\begin{equation}
{\partial}_t  \beta(u)+ b(y)\nabla_{y} \beta(u) =0
\end{equation}
for all continuously differentiable functions $\beta$, under some suitable growth assumptions for $\beta$. In the rest of the discussion we focus on bounded solutions, so that $\beta$ may be taken arbitrary. 

A major point of the DiPerna-Lions theory is that if $b$ is Sobolev then all weak solutions of
\eqref{e:example} are renormalized. 
If in addition $b$ has bounded divergence, the above fact implies that the solutions to \eqref{e:example} are unique by a simple Gronwall argument
applied to the absolute value of the difference of two solutions with the same initial data. This in turn has also a 
compactness ``effect'', cf.~\cite[Th. II.4]{DiLi1989}.
In order to understand the latter point, consider the solutions
to an approximate problem, again with Sobolev regularity of the coefficients:
\begin{equation}
(u_k)_t(t,y)+b_k(y)\nabla_y u_k(t,y)=0\, .
\end{equation}
Let us rewrite the equations for $\beta (u_k)$ in the distributional way described above:
\begin{equation}\label{tr-appr}
\partial_t\beta(u_k) +\div_{y}(b_k\beta(u_k))-\div_yb_k(t,y)\beta(u_k)=0.
\end{equation}
Assume that $b_k, \div_yb_k$ converge strongly in $L^1_{loc}$  to $b$ and $\div_yb$ respectively and let $k\to\infty$ in \eqref{tr-appr}. If the initial condition is bounded, the renormalized property and a simple comparison show that the solutions are uniformly bounded as well. Thus we conclude that $\beta(u_k)\weakstar\bar\beta$ in $L^\infty$ for some $\bar\beta$. Then 
$$\partial_t\bar\beta +\div (b\bar\beta)-\div_y b(y)\bar\beta=0$$
in a distributional sense.
Choose now $\beta(u)=u^2$, where $u$ is the solution for the limit problem \eqref{e:example}. Assuming renormalization for the limit we also have
 $$\partial_t (u^2) +\div (bu^2)-\div_y b(y)u^2=0.$$
By unique solvability we conclude that 
$$
\bar \beta=u^2$$
and hence from $u_k\weak u$ in $L^2_{loc}$ and $u_k^2 \weak u^2$ in $L^1_{loc}$  we are able to conclude that  $u_k\to u$ strongly in $L^2_{loc}$.

Notice that if we add a term  $\alpha u$, with $\alpha\in\R$, to  the right-hand side,  then choosing $\beta(u)=u^2$ we can use $u\beta'(u)=2\beta(u)$ and a similar scheme easily follows.  

Coming back to the original problem, if $u$ were a smooth solution to \eqref{E1}-\eqref{f}, then for all functions $\beta\in{\mathcal C}^1$ it would satisfy
\begin{equation}
{\partial}_t  \beta(u)+ b(t,x,r)\nabla_{x,r} \beta(u) =\beta'(u)\int_{\R^j}\gamma(t,x,r,\tilde r)u(t,x,\tilde r)\;d\tilde r\, ,
\end{equation}
which therefore is the natural ``renormalization condition'' for \eqref{E1}-\eqref{f}.
Unfortunately the above equation is not anymore an equation in $v:= \beta(u)$ and we cannot follow
the scheme above to infer the stability of \eqref{E1}-\eqref{E0}.

Significant simplifications would be provided if the integral operator were compact. In particular if the integration in \eqref{f} were with respect to all the variables (which is {\em not} the case here), then one could still apply the renormalization techniques with $\beta(u) = \exp (u)$. Note that for a weakly convergent  bounded sequence $(u^k)$ the product of the weakly convergent term $\exp(u^k)$ with a compact integral operator would converge giving the limiting identity
\begin{equation}
{\partial}_t  \bar \beta+ b(t,x,r)\nabla_{x,r} \bar\beta =\bar\beta\int_{\R^j\times\R^n}\gamma(t,x,y,\tilde r)u(t,y,\tilde r)\;d\tilde r\;dy. 
\end{equation}
Then by  unique solvability one could conclude that $\bar \beta=\exp(u)$, which together with the strict convexity of the function $\beta$ would provide the strong convergence of $u^k$ in $L^p$.

Thus, as the renormalization methods seem to fail for the general case considered here, we will direct our attention to the Lagrangian formulation. 
The scheme of DiPerna and Lions reduces the study of ODEs to that of transport equations. In 
 \cite{CrDL2008} (see also \cite{AmLeMa2005}) the authors have shown that many of the results proved by DiPerna and Lions can be recovered from a priori estimates in the Lagrangian formulation. This motivated us to reformulate \eqref{E1}-\eqref{f} in the Lagrangian setting, although ultimately we do not really need the estimates in \cite{CrDL2008}. In Section~\ref{RegLagrFl} we recall the definition of {\it regular Lagrangian flows} and present the advantages of this approach.

Notice that, if the Sobolev vector field $b$ has bounded divergence, the existence, uniqueness and stability of regular Lagrangian flows have been proved already in the seminal paper by DiPerna and Lions, \cite{DiLi1989}. Later Ambrosio extended the results to the important case of $BV$ vector fields with bounded divergence, see \cite{Am2004},~ \cite{Ja2010}. A further extension to the case
where the divergence is in $BMO$ is due to Mucha in \cite{Mu2010}. 

\subsection{Motivations}\label{Motivations}
We complete this section by recalling the system describing the flow of polymers which motivated us to the study \eqref{E1}-\eqref{f}. 

For  a given Lipschitz domain $\Omega \subset \mathbb{R}^d$ and a given time interval $(0,T)$, we consider the system consisting of balance of the linear momentum and the incompressibility constraint in the form of 
\begin{equation}\label{NS}
\begin{split}
\partial_t\bv(t,x)+ \div_x (\bv(t,x)\otimes \bv(t,x)) +\nabla_x q(t,x)- {\rm div}_x \bS(\tilde\psi(t,x),\bD_x \bv(t,x))&= \bef,\\
{\rm div}_x \bv(t,x) &= 0,
\end{split}
\end{equation}
where $\bv:Q \to \mathbb{R}^d$ is the velocity of the solvent, $q:Q\to \mathbb{R}$ is the pressure, $\bef:Q\to \mathbb{R}^d$ is the density of the external body forces, $Q:=(0,T)\times \Omega$. The  viscous  part of the Cauchy stress $\bS:Q\to \mathbb{R}^{d\times d}$ is given by the formula
\begin{equation} \label{T}
{\bS(\tilde{\psi}(t,x),\bD_x \bv(t,x))}:= \nu(\tilde\psi(t,x), |\bD_x\bv(t,x)|)\bD_x\bv(t,x),
\end{equation}
where $\bD_x \bv$ denotes the symmetric velocity gradient, i.e. $\bD_x \bv:=\frac12 (\nabla_x \bv + (\nabla_x \bv)^T)$, and the generalized viscosity $\nu:\R_+\times\R_+\to\R_+$ can  depend on the shear rate $|\bD_x\bv|$ and on
the averaged distribution function of polymers $\tilde\psi :Q\to \mathbb{R}_+$. The latter is given by the formula
\begin{equation}\label{psi-aver}
\tilde{\psi}(t,x):=\int_{\R_+} \alpha(r)\psi(t,x,r)\; dr,
\end{equation}
where $\R_+:=(0,\infty)$ and $\alpha:\R_0 \to \mathbb{R}_+$ is a continuous nonnegative function (a weight depending on the length $r$ of the polymer). The distribution function
$\psi:Q\times \R_0\to\R_+$  is assumed to satisfy the following equation
\begin{equation}\label{eq:psi}
\begin{split}
\partial_t \psi&(t,x,r) + \bv(t,x) \cdot \nabla_x\psi(t,x,r) +\tau(r) \phi(t,x)\partial_r\psi(t,x,r)
\\=& -\beta(r, \bv)\psi(t,x,r) +2\int_r^{\infty}\beta(\tilde r,\bv )\kappa(r,\tilde r)\psi(t,x,\tilde r)\,\dd\tilde r,
\end{split}
\end{equation}
in $Q\times\R_+$. Here  $\tau:[0,\infty)\to \mathbb{R}_+$ is the polymerization rate, $\beta:\R_+\times \mathbb{R}^d\times \mathbb{R}^{d\times d} \to \mathbb{R}_+$ is the fragmentation rate of polymers of size $r$, which can depend  also on the macroscopic quantities (namely on the velocity of the solvent)  and finally
$\kappa(r,\tilde r)$ denotes the probability that a polymer of length $\tilde{r}$ will split into two polymers of length $r$ and $\tilde r-r$.
The function $\phi$ appearing in the polymerization term is the concentration of free monomers and satisfies the equation
\begin{equation}\label{eq:phi}
\begin{split}
&\partial_t \phi(t,x)+\bv(t,x) \cdot \nabla_x\phi(t,x)
-A_0\Delta_x\phi(t,x)
=-\phi(t,x)\int_{0}^\infty\partial_r (r\tau(r))\psi(t,x,r)\,\dd r 
\end{split}
\end{equation}
in the space time cylinder $Q$. The rate of the diffusion  $A_0>0$ is a constant and the additional transport term is due to the viscosity of the  solvent.
The system is supplemented with appropriate initial and boundary conditions. More precisely we assume that the velocity satisfies the Navier slip boundary conditions, i.e.,
\begin{equation}
\begin{split}
\bv \cdot \bn &= 0 \textrm{ on }\partial \Omega,\\
(\bS \bn)_{\btau}&=-\alpha \bv \textrm{ on } \partial \Omega,
\end{split}\label{BC-v}
\end{equation}
where $\alpha\ge 0$ is the friction parameter, $\bn$ is the unit outward normal vector, and for any $\bv$ we denoted by $\bv_{\btau}:= \bv-(\bv\cdot \bn)\bn$ the projection onto the tangent hyperplane to the boundary. For $\psi$ and $\phi$, we prescribe the Neumann condition with respect to the $x$ variable as well, i.e.,
\begin{equation}
\begin{split}\label{BC-psi}
\nabla_x \psi \cdot \bn =0 \textrm{ on } \partial \Omega
\end{split}
\end{equation}
and 
\begin{equation}
\begin{split}\label{BC-phi}
\nabla_x \phi \cdot \bn =0 \textrm{ on } \partial \Omega
\end{split}
\end{equation}
and we assume that $\psi$ vanishes at infinity, namely
\begin{equation}
\begin{split}\label{BCr-psi}
\lim_{r\to \infty}\psi(t,x,r)=0.
\end{split}
\end{equation}
Finally,  we prescribe the following  initial conditions
\begin{equation}
\bv(0,x)=\bv_0(x) \textrm{ in } \Omega, \qquad \diver \bv_0=0 \textrm{ in } \Omega, \qquad \bv_0\cdot \bn = 0 \textrm{ on }\partial \Omega\label{IC-v}
\end{equation}
and 
\begin{align}
\psi(0,x,r)&=\psi_0(x,r) \textrm{ in } \Omega\times (r_0,\infty), &&\psi_0\ge 0.\label{IC-psi}
\end{align}

We assume that $\nu:\R_+\times\R_+\to\R_+$ is a continuous function such that, for some $p>\frac{2d}{d+2}$, the following three inequalities hold for all $\bxi, \tilde{\bxi}\in\R^{d\times d}$:
\begin{equation} \label{S}
\begin{split}
|{\bS(\cdot,\bxi)}|&\le K(1+|\bxi)^{p-1}\\
{\bS(\cdot,\bxi)}\cdot\bxi&\ge K^{-1}|\bxi|^p-K,\\
(\bS(\cdot, \bxi)-\bS(\cdot,\tilde{\bxi}))\cdot (\bxi - \tilde{\bxi})&> 0.
\end{split}
\end{equation} 

For a study of the model of concentrated polymers with diffusion term in the equation for polymer density we refer to~\cite{BuGwSuSw2014}. 
The corresponding theory for dilute polymers, namely the Navier-Stokes-Fokker-Planck system, has been studied both in the parabolic case (\cite{BaSu2011, BaSu2012a, BaSu2008, BaSu2012b}) and without diffusion, namely in the setting of a transport equation for the microscopic quantity, cf.~\cite{Ma2013}.

The relation between the model above and \eqref{E1}-\eqref{f} is given by the formulas $b(t,x,r)=(\bv(t,x), \tau(r)\phi(t,x))$ and $\gamma(t,x,r,\tilde r)=\beta(\tilde r,\bv)\kappa(r,\tilde r)$ where 
\begin{equation}
\label{df-kappa}
\kappa(r,\tilde{r}):=\left\{\begin{aligned}&\frac{1}{\tilde{r}} &&\textrm{if }  0<r<\tilde{r},\\
&0 &&\textrm{otherwise}.
\end{aligned}\right.
\end{equation}

\section{Regular Lagrangian flows}\label{RegLagrFl}
We wish to associate to each vector field $b$ a corresponding flow mapping $X_b:[0,T]\times\R^{n+j}\to\R^{n+j}$ which satisfies the following system of ODEs
\begin{equation}
\begin{split}
\frac{dX_b(t,x,r)}{dt}&=b(t,X_b(t,x,r)), \ t\in[0,T],\\
X(0,x,r)&=(x,r).
\end{split}
\end{equation}
Keeping in mind Assumption \ref{main_A},
in the sequel we will sometimes use the notation
\begin{equation}\label{X1X2}
\begin{split}
&X_b(t,x,r)=(X_b^{(1)}(t,x), X_b^{(2)}(t,x,r)) 
\end{split}\end{equation}
where
\begin{equation}\begin{split}
X_b^{(1)} & \mbox{   is the projection of $X_b$ on the first } n  \mbox{ components,}\\ 
X_b^{(2)} & \mbox{  is the projection of $X_b$ on the last } j \mbox{ components.} 
\end{split}
\end{equation}
As a consequence of (B1) in Assumption \ref{main_A} the flow $X_b^{(1)}$ will then be independent of $r$. 

\begin{definition}\label{d:Lagr_flow}
Following the DiPerna-Lions theory we shall say that $X_b:\R_+\times\R^{n+j} \to\R^{n+j}$ is a {\it regular Lagrangian flow } for $b\in L^1(0,T;
W^{1,q}_{loc}(\R^{n}\times\R^j))^{n+j}$ if the following two conditions are satisfied:
\begin{enumerate}
\item[$(i)$] For a.a. $(x,r)\in\T\times\R^j$  the mapping $X_b (\cdot, x ,r):[0,T]\to\R^{n+j}$ is an absolutely continuous integral solution of 
$\dot{a}(t)=b(t,a(t))$  (for $t\in[0,T]$) with the initial condition $a(0)=(x,r)$.  

\item[$(ii)$] Let $\mu_t = (X_b (t, \cdot))_\sharp \mathcal{L}^{n+j}$ (where $\mathcal{L}$ denotes the Lebesgue measure) namely 
\begin{equation}\label{mu}
\mu_t(A)={\mathcal L}^{n+j}(X^{-1}_b(t,A))\quad\mbox{for every Borel set}\ A\subset\R^{n+j}. 
\end{equation}
Then there exists a constant $L$ (which from now on will be called {\em incompressibility constant}) such that
\begin{equation}\label{measures}
\mu_t(A)\le L {\mathcal L}^{n+j}(A) \quad\mbox{for every Borel set}\ A\subset\R^{n+j}. 
\end{equation}
\end{enumerate}
\end{definition}

We recall the following corollary of the fundamental theory contained in \cite{DiLi1989}.

\begin{theorem}\label{t:DL}
Assume $b$ satisfies (B2) and (B3) of Assumption \ref{main_A} and the bound $\div\, b\in L^1 ((0,T), L^\infty (\R^n \times \R^j)$. 
Then there is a unique regular Lagrangian flow which in addition has the bounds
\begin{equation}
e^{-\int_0^t \|\div\, b (s, \cdot)\|_\infty} \mathcal{L}^{n+j} \leq \mu_t \leq e^{\int_0^t \|\div\, b (s, \cdot)\|_\infty} \mathcal{L}^{n+j}.
\end{equation}
\end{theorem}

Since in our case we have also the structural hypothesis (B1) of Assumption \ref{main_A} we can apply Theorem \ref{t:DL} to the system of ODEs for $X_b^{(1)}$ and to that for $X_b$ separately. From the uniqueness part of Theorem \ref{t:DL} it is then straightforward to infer that 
 \begin{equation}
 t\mapsto X_b(t,x,r)=(X_b^{(1)}(t,x),X_b^{(2)}(t,x,r))\, .
 \end{equation}
Moreover, by the bounds on the incompressibility constant we also infer the existence of densities $\varrho_1$ and $\varrho$ for the absolutely continuous
measures $\mu^{(1)}_t := (X_b^{(1)} (t, \cdot))_\sharp \mathcal{L}^n$ and $\mu_t = (X_b (t, \cdot))_\sharp \mathcal{L}^{n+j}$, satisfying the bounds
\begin{align}
& e^{-\int_0^t\|\div_x b^{(1)}\|_{L^\infty}\;ds}\le \varrho_1(t,x)\le e^{\int_0^t\|\div_x b^{(1)}\|_{L^\infty}\;ds},\label{rho1}\\
& e^{-\int_0^t\|\div_{(x,r)}b\|_{L^\infty}\;ds}\le \varrho(t,x,r)\le e^{\int_0^t\|\div_{(x,r)}b\|_{L^\infty}\;ds}\label{rho}.
\end{align}
Such densities satisfy the following continuity equations
\begin{equation}\label{varrho1}
\begin{split}
\partial_t \varrho_1(t,x)+\div_{x}(b^{(1)}(t,x)\varrho_1(t,x))=0,\\
\varrho_1(0,x)=1,
\end{split}
\end{equation} 
\begin{equation}
\begin{split}
\partial_t \varrho(t,x,r)+\div_{(x,r)}(b(t,x,r)\varrho(t,x,r))=0,\\
\varrho(0,x,r)=1.
\end{split}\label{cont-equ}
\end{equation}
Moreover, the very definition of the measures $\mu^{(1)}_t$ and $\mu_t$ give the following ``change of variables formulas'', valid for every
bounded test functions $\varphi$ with bounded support:
\begin{equation}\label{e:first}
\int_{\R^n} \varphi(t,x)\varrho_1 (t,x)\; dx=\int_{\R^n}\varphi(t, X_b^{(1)}(t,x))\;dx,
\end{equation}
\begin{equation}\label{e:second}
\int_{\R^n\times\R^j} \varphi(t,x,r)\varrho (t,x,r)\; dx\;dr=\int_{\R^n\times\R^j}\varphi(t, X_b^{(1)}(t,x), X_b^{(2)}(t,x,r))\;dx\;dr\, .
\end{equation}

Observe that condition~(B3) in Assumption \ref{main_A} does not imply the boundedness of the flow. This needs a minor technical adjustment in order to bound the ``inflow'' of trajectories in a ball $B_R$ at a given time. Following the same arguments of \cite[Prop. 3.2]{CrDL2008} we reach the following lemma:
 \begin{lemma}\label{infty}
Assume that $b=b_1+b_2$ as in condition (B3) of Assumption \ref{main_A} and that $\div\, b\in L^1(0,T;L^\infty(\R^{d}))$. Let $X_b$ be the flow of $b$ and $R>\rho$. Then
\begin{equation}
|\{x\notin B_R(0): X_b(t,x)\in B_\rho(0)\}|\le \frac{C(\|\div\, b\|_{L^1(0,T;L^\infty(\R^{d}))}, \|b_1\|_{L^\infty},\|b_2\|_{L^1},t)}{R}.
\end{equation}
\end{lemma}

We will also recall the result on stability of the flows.

\begin{theorem}\cite[Theorem 7.5]{Ambrosio}\label{stabilityX}
Let $b$ be a vector field satisfying Assumption \ref{main_A} and let $b^h$ be a sequence of vector fields for which (B3) in Assumption \ref{main_A} holds with a decomposition $b^h_1+ b^h_2$ where $b^h_1/(1+|x|+|r|)$ is equibounded and equiintegrable in $L^1 ([0,T]\times \R^d)$, whereas $b^h_2/(1+|x|+|r|)$ is equibounded in $L^1 ((0,T), L^\infty)$. If $b^h_1\to b_1$ and $b^h_2\to b_2$ a.e., then the corresponding flows $X_{b^h} (t, \cdot)$ converge in measure to $X_b (t, \cdot)$.
\end{theorem}

Indeed, if $q>1$ in Assumption \ref{main_A} (B2), then it is possible to give a precise rate of convergence, cf. \cite[Lemma 6.4]{CrDL2008}. The latter is however not really needed in our proof.

Note that in our setting the regular Lagrangian flow $X_b$ has a well defined inverse for each $t$.
Namely for every $X_b (t, \cdot)$ there is a well defined $X_b^{-1} (t, \cdot)$ such that $X_b (t, X_b^{-1} (t,x)) = X_b^{-1} (t,X_b (t,x)) =x$
for a.e. $x$. Under our assumption the map $X_b^{-1} (t, \cdot)$ is also nearly incompressible and indeed we have the
following simple formula. If we consider the flow
\begin{equation}
\left\{
\begin{array}{l}
\frac{d Y (s,x)}{ds} = b (s, Y (s,x)) \\ \\
Y (t, x) =x
\end{array}\right.
\end{equation}
then $X_b^{-1} (t,x) = Y (0,x)$. 
For this fact the reader might consult one of the references \cite{Ambrosio}, \cite{AC} or \cite{DL}. Using Theorem \ref{stabilityX} and the above ``backward'' ODE, it is easy to see that under the assumptions of Theorem \ref{stabilityX} we conclude the convergence in measure of the maps $X_{b_h}^{-1} (t, \cdot)$ to $X_b^{-1} (t, \cdot)$.

Having these results we are able to prove the following appropriate change of variables formula, which will be used to recast
\eqref{E1}-\eqref{f} in a ``Lagrangian form''. 
   \begin{lemma}\label{l:change}
Assume $b$ satisfies Assumption \ref{main_A} and let 
\[
X_b (t,x,r) = (X_b^{(1)}(t,x), X_b^{(2)}(t,x,r))
\] 
be the regular Lagrangian flow according to  \eqref{X1X2}. Moreover, let $\varrho_1(t,x)$ and $\varrho(t,x,r)$ be the densities of the flows 
 $X_b^{(1)}$ and $X_b$ respectively and define 
\[
\varrho_2(t,x,r):=\frac{\varrho_1 (t, X_b^{(1)} (t,x))}{\varrho(t,X_b (t,x,r))}\, .
\]
Then for all $t\in[0,T]$
 \begin{equation}
 \int_{\R^j}\varphi(t,X_b(t,x,r),\tilde r)\;d\tilde r=\int_{\R^j}\varphi(t,X_b(t,x,r), 
 X_b^{(2)}(t,x,\tilde r))\varrho_2(t,x,r)\;d\tilde r
 \end{equation}
 for every $\varphi\in L^1([0,T]\times \R^n\times\R^j\times \R^j)$.
 \end{lemma}
 
 \begin{proof}
Assume for the moment that $b$ is smooth and bounded, so that the flows and their inverses map bounded sets into bounded sets
at every finite time. If we use \eqref{e:second} with $\varphi (t,x, r) = \psi (t,x, r)/\varrho (t,x, r)$ we reach the identity
\begin{align*}
\int \psi(t,x,r)\, dx\, dr &= \int \psi (t, X_b (t,x,r)) \varrho^{-1} (t, X_b (t,x,r))\, dx\, dr\, .
\end{align*}
Similarly, testing \eqref{e:first} with $\varphi (t,x) = \psi (t, [X_b^{(1)}]^{-1} (t,x))$ we reach the identity
\[
\int \psi (t, [X_b^{(1)}]^{-1} (t,x)) \varrho_1 (t,x) \, dx = \int \psi (t,x)\, dx\, .
\]
Consider now the $n$-dimensional ball $B_\varepsilon(x)$ centered at $x$ with radius $\varepsilon$. Using the Lebesgue differentation theorem and changing variables twice according to the rules above we achieve:
 \begin{equation}\begin{split}
 \int_{\R^j}&\varphi(t,X_b(t,x,r),\tilde r)\;d\tilde r 
 =\lim_{\varepsilon\to0}\frac{1}{{\mathcal L}^n(B_\varepsilon(x))}\int_{B_\varepsilon(x)}\int_{\R^j}\varphi(t, X_b^{(1)}(t,x), X_b^{(2)}(t,x,r),\tilde r)\;d\tilde r dx\\
 &=\lim_{\varepsilon\to0}\frac{1}{{\mathcal L}^n(B_\varepsilon(x))}\int_{X_b^{(1)}\left(t,B_\varepsilon(x)\right)}\int_{\R^j}
 \varphi(t,\bar x, X_b^{(2)}(t,[X_b^{(1)}]^{-1}(t,\bar x),r),\bar r) \varrho_1(t,\bar x) \;d\bar r \;d\bar x\\
 &=\lim_{\varepsilon\to0}\frac{1}{{\mathcal L}^n(B_\varepsilon(x))}\int_{B_\varepsilon(x)}\int_{\R^j}
 \varphi(t,X_b^{(1)}(t,x), X_b^{(2)}(t,x,r),X_b^{(2)}(t,x,\tilde r))\frac{\varrho_1 (t,X_b^{(1)} (t,x))}{\varrho (t, X_b (t,x,r))} \;d\tilde r \;d x\\
 &=\int_{\R^j}
 \varphi(t,X_b^{(1)}(t,x), X_b^{(2)}(t,x,r),X_b^{(2)}(t,x,\tilde r)) \rho_2 (t,x,r) \;d\tilde r.
 \end{split} \end{equation}

The argument above is only ``formal'' in our case because the test functions $\psi$ used in the above formulas do not have
bounded support. However, to justify the computations,
we take a sequence of regularizations $b_k$ of $b$ bounded in $L^1(0,T;W^{1,\infty}(\R^n\times\R^j)\cap {\mathcal C}^1(\R^n\times\R^j))^{n+j}$ and satisfying the structural Assumption \ref{main_A}, the conditions in Theorem \ref{stabilityX} and the requirement that ${\rm div}\, b_k \to {\rm div}\, b$ strongly in $L^1_{loc}$. 

 Again decomposing $b_k(t,x,r)=(b_k^{(1)}(t,x), b_k^{(2)}(t,x,r))$ by $\varrho^k_1(t,x)$ we mean the density of the flow of $b_k^{(1)}$ solving the continuity equation
 \begin{equation}\label{varrho1k}
\begin{split}
\partial_t \varrho_1^k(t,x)+\div_{x}(b_k^{(1)}(t,x)\varrho^k_1(t,x))=0,\\
\varrho^k_1(0,x)=1,
\end{split}
\end{equation} 
and by 
$\varrho^k(t,x,r)$  the density of the flow of $b_k$ solving 
 \begin{equation}\label{varrho1}
\begin{split}
\partial_t \varrho^k(t,x,r)+\div_{x,r}(b_k(t,x,r)\varrho^k(t,x,r))=0,\\
\varrho^k(0,x,r)=1.
\end{split}
\end{equation} 

By Theorem~\ref{stabilityX} we obtain that $X_{b_k}$ - the regular Lagrangian flow for $b_k$,  converges locally in measure to $X_{b}$. 
 First we shall comment on the convergence of the density $\varrho_2^k(t,x,\tilde r)=
\frac{\varrho_1^k(t,X_b^{(1)}(t,x))}{\varrho^k(t, X_b (t,x,r))} $. 

The condition \eqref{rho1} provides that $\varrho$ is bounded away from zero and by standard stability arguments for the continuity equation, both 
 $\varrho_1^k$ and $\varrho^k$ converge almost everywhere to $\varrho_1$ and 
 $\varrho$ respectively. Indeed the DiPerna-Lions theory ensures that the solutions to the continuity equations are stable under
the convergence in Theorem \ref{stabilityX}, cf. \cite{AC}. Note, moreover, that we have the renormalized property for solutions of the continuity equation in the following form:
\begin{align}
&\partial_t \beta (\varrho^k) + {\rm div}\, (\beta (\varrho^k) b_k) = (\beta (\varrho^k) - \rho \beta' (\varrho^k)){\div}\, b_k,\\
&\partial_t \beta (\varrho) + {\rm div}\, (\beta (\varrho) b) = (\beta (\varrho) - \varrho \beta' (\varrho)){\div}\, b\, ,
\end{align}
(cf. \cite[Theorem 24]{AC}). Since the divergence of the vector fields converge locally strongly, from the uniqueness
of the solution to the continuity equation we conclude that $\beta (\rho^k)$ converges weakly to $\beta (\rho)$. The arbitrariness of
the test function $\beta \in C^1$ gives then strong $L^1_{loc}$ convergence.

We next show that $\varrho^k_1(t,X_{b_k}^{(1)}(t,x))$ converges almost everywhere to $\varrho_1(t,X_b^{(1)}(t,x))$. Fix a ball $B_\rho$ and a much larger ball $B_R$. If we give up a set of small measure $K\subset B_{R+1}$, we can assume that $\varrho_1$ is continuous on the complement. We can then extend it continuously to a new function $\hat\varrho$ on $B_{R+1}$ and, multiplying by a cut-off function which is identically $1$ on $B_R$, we can assume that $\hat\varrho$ vanishes identically outside of $B_{R+1}$. We can then
assume that
$\|\hat\varrho (t, \cdot)-\varrho_1 (t, \cdot)\|_{L^1 (B_R)}$ is small and using Lemma \ref{infty} and the near incompressibility, we can assume that $\|\hat\varrho (t, (X_{b_k}^{(1)} (t, \cdot)) - \varrho_1 (t, X_{b_k}^{(1)} (t, \cdot))\|_{L^1 (B_\rho)}$ is also small (depending on $R$).
However $\hat\varrho (t, X_{b_k}^{(1)} (t, \cdot))$ converges pointwise a.e. to $\hat\varrho (t, X_b^{(1)} (t, \cdot))$ and thus in $L^1$ by the Lebesgue dominated convergence theorem. Letting first the size of $|K|$ go to $0$ and then $R$ go to infinity, we conclude that
$\varrho_1 (t, X_{b_k}^{(1)} (t, \cdot))$ converges strongly in $L^1_{loc}$ to $\varrho_1 (t, X_b^{(1)} (t, \cdot))$. 

We next need to estimate $\|\varrho^k_1 (t, X_{b_k}^{(1)} (t, \cdot))- \varrho_1 (t,X_{b_k}^{(1)} (t, \cdot))\|_{L^1 (B_\rho (0))}$. We now use the near incompressibility of the flow $X_{b_k}$ to write down
\begin{align*}
&\int_{B_\rho} |\varrho^k_1 (t, X_{b_k}^{(1)} (t, x))- \varrho_1 (t, X_{b_k}^{(1)} (t,x))|\, dx \leq  C\int_{X_{b_k}^{(1)} (t, \cdot)^{-1} (B_\rho (0))}
|\varrho^k_1 (t,y) - \varrho_1 (t,y)|\, dy\\
\leq\; & C\int_{(X_{b_k}^{(1)} (t, \cdot)^{-1} (B_\rho (0)))\cap B_R} |\varrho^k_1 (t,y) - \varrho_1 (t,y)|\, dy + C |(X_{b_k}^{(1)} (t, \cdot)^{-1} (B_\rho (0)))\setminus B_R|\, .
\end{align*}
Using Lemma \ref{infty} we conclude that the second summand can be made smaller than any $\varepsilon >0$ by choosing $R$
large enough. On the other hand the first summand is bounded by $\|\varrho_1^k - \varrho_1\|_{L^1 (B_R)}$, which converges to $0$ as $k\to \infty$. 

This completes the proof of the convergence of $\varrho^k_1(t,X_{b_k}^{(1)}(t,x))$ to $\varrho_1(t,X_b^{(1)}(t,x))$.
The convergence of the other terms follow in a very similar fashion and we omit the corresponding details.
 \end{proof}
 
Our strategy in the proof of Theorem~\ref{Th:stabilityI}  will follow upstream and duethe approach of DiPerna and Lions: we will start with the PDE problem and turn to the corresponding Lagrangian formulation. More precisely we need the following 

\begin{lemma}\label{l:lagrangian}
Let $b$ and $\gamma$ be as in Theorem \ref{Th:stabilityI} and let $X_b = (X^{(1)}_b, X^{(2)}_b)$ be the regular Lagrangian flow of
$u$. Then $u\in L^\infty (0,T; L^p (\R^n\times \R^j))$ solves \eqref{E1} if and only if 
\begin{equation}\label{u_tilde}
\tilde u(t,x,r):=u(t,X_b^{(1)}(t,x), X_b^{(2)}(t,x,r)).
\end{equation}
solves the equation
\begin{equation}\label{char}
\tilde{u} (t,x):=u_0(x,r)+\int_0^t\int_{\R^j}\gamma(s,X_b^{(1)}(s,x), X_b^{(2)}(s,x,r),X_b^{(2)}(s,x,\tilde r))\varrho_2(s,x,\tilde r)
\tilde u(s,x,\tilde r)\;d\tilde r.
\end{equation}
\end{lemma}

\begin{remark}\label{r:makes_sense}
It is interesting to notice that the formulation \eqref{A} of the problem makes sense even without the assumption $b\in L^1 (0,T; L^{p'}_{loc})$. 
\end{remark}

\begin{proof}
The lemma is split in two steps. First of all we show the following. If 
\begin{itemize}
\item $f\in L^\infty (0,T; L^p)$;
\item $b\in L^1 (0, T; L^{p'}_{loc})$ and satisfies Assumption \ref{main_A},
\end{itemize}
then $u\in L^\infty (0, T; L^p)$ is a solution of 
\[
\partial_t u + b \cdot \nabla_{x,r} u = f
\]
if and only if $\tilde{u} (t,x,r):= u (t, X_b (t,x,r))$ solves
\[
\partial_t \tilde{u} (t,x,r) = f (t,X_b (t,x,r))\, .
\]
This claim is obvious if $b$ and $f$ are smooth, as we can use the chain rule for derivatives. In order to cover the most general case we can argue as in the proof of Lemma \ref{l:change} and use an approximation procedure.

\medskip

We now apply the identity above to concude that \eqref{E1} is equivalent to 
\begin{equation}\label{ode-}
\partial_t \tilde u(t,x,r)=\int_{\R^{j}}\gamma(t,X_b^{(1)}(t,x), X_b^{(2)}(t,x,r),\tilde r)u(t,X_b^{(1)}(t,x),\tilde r)\;d\tilde r
\end{equation}
and after the change of variables the right-hand side gives the following expression
\begin{equation}
\int_{\R^{j}}\gamma(t,X_b^{(1)}(t,x), X_b^{(2)}(t,x,r),X_b^{(2)}(t,x,\tilde r)) u(t,X_b^{(1)}(t,x),X_b^{(2)}(t,x,\tilde r))\varrho_2(t,x,\tilde r)\;d\tilde r\, .
\end{equation}
Due to \eqref{u_tilde} we will finally reach
\begin{equation}\label{ode}
\begin{split}
\partial_t \tilde u(t,x,r)
=\int_{\R^{j}}\gamma(t,X_b^{(1)}(t,x), X_b^{(2)}(t,x,r),X_b^{(2)}(t,x,\tilde r)) \tilde u(t,x,\tilde r)\varrho_2(t,x,\tilde r)\;d\tilde r.
\end{split}
\end{equation}
\end{proof}


\section{Proof of Theorem \ref{Th:stabilityI}}

\subsection{ Integral form of the equation}
For $1\leq p < \infty$ we 
define an operator 
\begin{equation}
A:{\mathcal C}([0,T_0];L^p(\R^n\times\R^j))\to {\mathcal C}([0,T_0];L^p(\R^n\times\R^j))
\end{equation}
as follows
\begin{equation}\label{A}
(A\tilde u)(t):=u_0(x,r)+\int_0^t\int_{\R^j}\gamma(s,X_b^{(1)}(s,x), X_b^{(2)}(s,x,r),X_b^{(2)}(s,x,\tilde r))\varrho_2(s,x,\tilde r)
\tilde u(s,x,\tilde r)\;d\tilde r.
\end{equation}
Then 
solving  \eqref{char} is equivalent to finding the fixed point of $A$. 
 We define $X_{b_k}$ as a regular Lagrangian flow for $b_k$ and 
 an operator $$A_k: {\mathcal C}([0,T_0];L^p(\R^n\times\R^j))\to {\mathcal C}([0,T_0];L^p(\R^n\times\R^j))$$
 as follows
\begin{equation}\label{Ak}
(A_k\tilde u)(t):=u_0(x,r)+ \int_0^t\int_{\R^{j}}\gamma(s,X_{b_k}^{(1)}(s,x), X_{b_k}^{(2)}(s,x,r), X_{b_k}^{(2)}(s,x,\tilde r))\varrho_2^k(s,x,\tilde r)\tilde u(s,x,\tilde r)\;d\tilde r.
\end{equation}

\begin{lemma}\label{op_conv}
Let the operators $A$ and $A_k$ be defined by \eqref{A} and \eqref{Ak} respectively. Then for all $\omega\in {\mathcal C}([0,T_0];L^p(\R^n\times\R^j))$ we have
\begin{equation}\label{convA}
{\|}A_k\omega- A\omega{\|}_{{\mathcal C}([0,T_0];L^p(\R^n\times\R^j))}\to 0
\end{equation}
as $k\to\infty$.  In particular $\sup_k \|A_k\| < \infty$.
In fact, if $T_0$ is sufficiently small, we have $\sup_k \sup_{\|u-v\|\leq 1} \|A_k (u - v)\|_{ {\mathcal C}([0,T_0];L^p(\R^n\times\R^j))}\leq \frac{1}{2}$.
\end{lemma}
\begin{proof}
The proof will consist of two steps, which refer to the assumptions of the Arzel\'a-Ascoli theorem. For this purpose let us first introduce the notation
\begin{equation}
J_k (t,x,r) :=\int_{\R^{j}}\gamma(t,X_{b_k}^{(1)}(t,x), X_{b_k}^{(2)}(t,x,r),X_{b_k}^{(2)}(t,x,\tilde r))\omega(t,x,\tilde r)\varrho_2^k(t,x,\tilde r)\;d\tilde r.
\end{equation}
In the first step we will check that the sequence of maps
\[
t \quad \mapsto \quad j_k (t,\cdot,\cdot):= \int_0^t J_k (\tau, \cdot,\cdot)\, d\tau
\]
is pointwise relatively compact in $L^p(\R^n\times\R^j)$. The second step is devoted to the equicontinuity of $j_k$. Using the change of variables formulas proved in the previous section we rewrite
\begin{equation}
\begin{split}
&\int_{\R^n\times\R^j}\Big{|}
\int_{\R^{j}}\gamma(t,X_{b_k}^{(1)}(t,x), X_{b_k}^{(2)}(t,x,r),X_{b_k}^{(2)}(t,x,\tilde r))\omega(t,x,\tilde r)\varrho_2^k(t,x,\tilde r)\;d\tilde r \Big{|}^p
\;dx\;dr\\
&=\int_{\R^n\times\R^j}\Big{|}
\int_{\R^{j}}\gamma(t,x, r,\tilde r)\omega\left(t,[X_{b_k}^{(1)}]^{-1}(t,x),[X_{b_k}^{(2)}]^{-1}(t,x,\tilde r)\right)\;d\tilde r\Big{|}^p\frac{1}{\bar \varrho^k(t,x, r)}
\;dx\;dr\\
&=\int_{\R^n\times\R^j}\Big{|}
\int_{\R^{j}}\gamma(t,x, r,\tilde r)\omega\left(t,[X_{b_k}^{(1)}]^{-1}(t,x),[X_{b_k}^{(2)}]^{-1}(t,x,\tilde r)\right)
{\frac{1}{\sqrt[p]{\bar \varrho^k(t,x, r)}}}
\;d\tilde r\Big{|}^p
\;dx\;dr\, .
\end{split}
\end{equation}
The special form of the densities $\bar \varrho^k$ is in fact not important: by arguments entirely similar to the ones of the previous section such functions are uniformly bounded and converge in measure to the corresponding density $\varrho$ appearing the
analogous identity for the final flow $b$.

Moreover, it is straightforward to check, using the change of variables formulas, the boundedness of the weights $\frac{1}{\bar{\rho}^k (t,x,r)}$ and the assumption on $\gamma$, that the bounds on operator norms $\|A_k\|$ in the last claims of the lemma follow easily.
{\it Step 1. Convergence in $L^p(\R^n\times\R^j)$ of $j_k$.}
As already observed, in order to get the convergence in measure of
\begin{equation}\label{464}
\omega\left(t,[X_{b_k}^{(1)}]^{-1}(t,x),[X_{b_k}^{(2)}]^{-1}(t,x,\tilde r)\right)
{\frac{1}{\sqrt[p]{\bar \varrho^k(t,x, r)}}}
\end{equation}
we shall concentrate on $\omega$. By Lemma~\ref{stabilityX} we obtain that $[X_{b_k}]^{-1}$ - the inverse of the regular Lagrangian flow for $b_k$ -  converges locally in measure to $[X_{b}]^{-1}$. With help of  Lusin's theorem and Lemma~\ref{infty} we show that 
\begin{equation}
\omega\left(t,[X_{b_k}^{(1)}]^{-1}(t,x),[X_{b_k}^{(2)}]^{-1}(t,x,\tilde r)\right) \to \omega\left(t,[X_{b}^{(1)}]^{-1}(t,x),[X_{b}^{(2)}]^{-1}(t,x,\tilde r)\right)
\end{equation} 
in measure. Indeed, let $\omega_R:=\omega\cdot\chi_{[0,R]\times B^n_{R}(0)\times B^{j}_{R}(0)}$ and 
$\hat\omega_R\in {\mathcal C}_c([0,T]\times\R^n\times\R^j)$ be such that for every $\varepsilon$ we have
$\|\omega_R-\hat\omega_R\|_{L^1}<\varepsilon$.
 Above, by $B^n_{R}(0)$ we mean $n-$dimensional ball and by $B^{j}_{R}(0)$ we mean $j-$dimensional ball. Note that if $[X_{b_k}^{(2)}]^{-1}(t,x,\tilde r)\in B_R(0),$ then  either $\tilde r\in B_\rho(0)$ or
 in the opposite case by Lemma~\ref{infty} we can estimate the measure of the set of such $\tilde r$ and show that it vanishes for large $\rho$. 
 To conclude  we have to show the uniform integrability in $L^p$ of \eqref{464}. Again we need to concentrate on $\omega$ as $\bar \varrho^k$ is uniformly bounded from above and from below. To prove our claim we will use the well known Orlicz-type condition
 which is equivalent to the uniform integrability, namely the existence of some convec function $g:\R\to\R$ such that
 $\lim_{|\xi|\to\infty}\frac{g(\xi)}{|\xi|^p}=0$ and 
 $g\left(\omega\left(t,[X_{b_k}^{(1)}]^{-1}(t,x),[X_{b_k}^{(2)}]^{-1}(t,x,\tilde r)\right)\right)$ is integrable. Obviously $\omega(t,x,\tilde r)$ is uniformly integrable in $L^p(\R^n\times\R^j)$, hence $g(\omega(t,x,\tilde r))$ is integrable for any function $g$ satisfying the above conditions. Notice that the change of variables yields
   \begin{equation}
 \begin{split}
& \int_{\R^n\times\R^j}
\int_{\R^{j}}g\left(\omega\left(t,[X_{b_k}^{(1)}]^{-1}(t,x),[X_{b_k}^{(2)}]^{-1}(t,x,\tilde r)\right)\right)
\;d\tilde r
\;dx\;dr\\
&= \int_{\R^n\times\R^j}
\int_{\R^{j}}g\left(\omega\left(t,x,\tilde r)\right)\right)\frac{1}{\bar \varrho_k(s,x,\tilde r)}
\;d\tilde r
\;dx\;dr.
\end{split}
 \end{equation}
 As $\bar \varrho^k$ is uniformly bounded from above and from below, the term on the right-hand side is finite. Therefore we conclude the uniform integrability of $\omega\left(t,[X_{b_k}^{(1)}]^{-1}(t,x),[X_{b_k}^{(2)}]^{-1}(t,x,\tilde r)\right)$. Thus by Vitali's theorem $j_k (t,\cdot,\cdot)$ converges strongly to $j(t,\cdot,\cdot)=\int_0^t J (\tau, \cdot, \cdot)\;d\tau$ in $L^p(\R^n\times\R^j)$, where 
 \begin{equation}
J (t,x,r) :=\int_{\R^{j}}\gamma(t,X_{b}^{(1)}(t,x), X_{b}^{(2)}(t,x,r),X_{b}^{(2)}(t,x,\tilde r))\omega(t,x,\tilde r)\varrho_2(t,x,\tilde r)\;d\tilde r.
\end{equation}

 {\it Step 2.  Equicontinuity in time}. To show that $j_k$
is equicontinuous we again proceed with the change of variables  and estimate
 \begin{equation}
\begin{split}
\int_{\R^n\times\R^j}&\Big{|}
\int_\tau^t\int_{\R^{j}}\gamma(s,X_{b_k}^{(1)}(s,x), X_{b_k}^{(2)}(s,x,r),\tilde r)\omega(s,x,[X_{b_k}^{(2)}]^{-1}(s,x,\tilde r))\;d\tilde r \;ds\Big{|}^p
\;dx\;dr\\
&\le \int_{\R^n\times\R^j}\|\omega(\cdot, x,[X_{b_k}^{(2)}]^{-1}(\cdot,x,\cdot))\|_{L^\infty(0,T;L^{p}(\R^j))}\\
&\quad\cdot\Big{|}\int_\tau^t\int_{\R^{j}}|\gamma(s,X_{b_k}^{(1)}(s,x), X_{b_k}^{(2)}(\cdot,x,r),\tilde r)|^{p'}\;d\tilde r\;ds\Big{|}^p\;dx\;dr.
\end{split}\end{equation}
Finally it is enough to argue for the continuity in time of the following term, which we then estimate with help of Jensen's inequality
and the incompressibility of the flow
\begin{equation}\label{equi}
\begin{split}
\int_{{\R^n}}\int_{\R^j}&\|\omega(\cdot, x,[X_{b_k}^{(2)}]^{-1}(\cdot,x,\cdot))\|_{L^\infty(0,T;L^{p}(\R^j))}\\
&\quad\cdot\Big{|}\int_\tau^t\int_{\R^{j}}|\gamma(s,X_{b_k}^{(1)}(s,x), X_{b_k}^{(2)}(\cdot,x,r),\tilde r)|^{p'}\;d\tilde r\;ds\Big{|}^p\;dx\;dr\\
&\le|t-\tau|^{p-1}\int_{\R^j\times\R^n}\|\omega(\cdot, x,[X_{b_k}^{(2)}]^{-1}(\cdot,x,\cdot))\|_{L^\infty(0,T;L^{p}(\R^j))}\\
&\quad\cdot\int_\tau^t\Big{|}\int_{\R^{j}}|\gamma(s,X_{b_k}^{(1)}(s,x), X_{b_k}^{(2)}(\cdot,x,r),\tilde r)|^{p'}\;d\tilde r\;\Big{|}^{p}ds\;dx\;dr\\
&\le C |t-\tau|^{p-1}\int_{\R^j\times\R^n}\|\omega(\cdot, x,[X_{b_k}^{(2)}]^{-1}(\cdot,x,\cdot))\|_{L^\infty(0,T;L^{p}(\R^j))}\\
&\quad\cdot\int_\tau^t\Big{|}\int_{\R^{j}}|\gamma(s,x,r,\tilde r)|^{p'}\;d\tilde r\;\Big{|}^{p}\;ds\;dx\;dr\\
&\le C |t-\tau|^{p-1}\|\omega\|_{L^\infty(0,T;L^{p}(\R^j\times\R^n))}
\cdot \sup_{x\in\R^n}\int_0^T\int_{\R^j}\Big{|}\int_{\R^{j}}|\gamma(s,x,r,\tilde r)|^{p'}\;d\tilde r\;\Big{|}^{p}\;dr\;ds.\\
\end{split}
\end{equation}
The above estimate completes the proof of {\it Step 2}. 
Combining both steps we finish the proof of the Lemma. 
\end{proof}
\begin{remark}
Note that if we could write the problem in the form
\begin{equation}\label{ode2}
\frac{d\tilde u }{dt}=\hat A[\tilde u]
\end{equation}
with  some integral operator $\hat A=\int_{\R^j} \hat \gamma(x,r,\tilde r)\; d\tilde r$
with an autonomous kernel $\hat \gamma$ (namely independent of time), then the solution would be given by
\begin{equation}\label{expA}
\tilde  u(t)=u_0\exp\left(\int_0^t\hat A[u(s)]\;ds\right).
\end{equation}
However, in the present setting, where the kernel $\gamma$ depends also on the time variable, such an exponential formula is valid only if we have the commutation relations $\hat A(s)\circ \hat A(t)=
 \hat A(t)\circ\hat A(s)$ for every $t,s\in [0,T]$. 
\end{remark}
\begin{proof}[Proof of Theorem~\ref{Th:stabilityI}]
First of all assuming $T_0$ is sufficiently small, all the $A_k$'s are contractions and thus the existence and uniqueness statements
follow by classical arguments. Moreover, a simple continuation trick removes the smallness assumption on $T_0$ and proves the claims for all time. Fix now an initial datum $u_0$ and consider the corresponding solutions $u_k$. In order to show that $u_k$ converges to $u$, we first show that this is correct for $t\in [0, T_0]$, where we have assumed that $T_0$ is chosen so small that
all $A_k$ are contractions with contracting constant smaller than $\frac{1}{2}$. We then have that $u_k$ is the limit, for $n\to\infty$ of $A_k^n (u_0)$. But using the contraction property we also conclude that $\|A_k^{n+1} (u_0) - A_k^n (u_0)\|\leq \left(\frac{1}{2}\right)^n \|A_k (u_0) - u_0\| \leq \frac{C}{2^n}$, for a constant $C$ independent of $k$. In particular, for a sufficiently large $n$ we have
\[
\sup_k \|A_k^n (u_0) - u_k\| \leq \frac{\varepsilon}{3}\, . 
\]
Having fixed such an $n$ we can use Lemma \ref{op_conv} to conclude that, for a sufficiently large $k$,
$\|A_k^n (u_0) - A^n (u_0)\| \leq \frac{\varepsilon}{3}$. We then conclude
\[
\|u - u_k\| \leq \|u - A^n (u_0)\| + \|A^n (u_0) - A_k^n (u_0)\| + \|A_k^n (u_0) - u_k\| \leq \varepsilon\, .
\]
This proves the convergence of the solutions on the interval $[0, T_0]$. Knowing now that $u_k (T_0, \cdot)$ converges to $u (T_0, \cdot)$, it is easy to use the same argument and extend the convergence to later times.
\end{proof}

\section{A remark on the strong convergence for the continuity equation}

We  end our paper with a short discussion on the conditions on $b_k$ in order to achieve stability of solutions. To show stability for the solutions of homogeneous transport equation, namely \eqref{E1}  with $f\equiv0$, it is not necessary to assume that the sequence $(\div\, b_k)_{k\in\N}$ converges strongly, see e.g.~\cite{CrDL2008}. Note instead that in our proof we need the strong convergence of the solutions to the continuity equation. A natural question is whether such strong convergence can be concluded even when $\div\, b_k$ converges weakly to $\div\, b$. Indeed the same question has arisen in a conversation of the authors with Eduard Feireisl. The following example provides a negative answer.




Consider $b_k:\R\to\R$, $b_k(x)=\frac{1}{k}\sin kx$. Then $b_k\to0$ strongly in the supremum norm. The divergence, which in this case is just $b_k'(x)=\cos kx$, is uniformly bounded, but converges only weakly to $0$ in 
$L^1_{loc}$.
Let $X_{b_k}:\R\times\R\to\R$ be the solution to 
\begin{equation}\label{bk-ode}
\begin{split}
\frac{\partial X_{b_k}(t,x)}{\partial t}&=b_k(t,X_{b_k}(t,x))\\
X(0,x)&=x.
\end{split}
\end{equation}
and $X_b(x,t)=x$ the static solution to the limit problem.
Then we obviously have that
\begin{equation}\label{0}
X_{b_k}\to X_b \mbox{ and }X_{b_k}^{-1}(\cdot,t)\to X_b^{-1}(\cdot,t)
\end{equation}
uniformly, but we will show that 
\begin{equation}\label{1}
\partial_xX_{b_k} \mbox{ fails to converge strongly in } L^1_{loc}.
\end{equation} 
Let $\varrho_k$ 
be the unique solution of the problem 
\begin{equation}
\begin{split}
\partial_t \varrho_k(t,x)+\div_{x}(b_k(x)\varrho_k(t,x))=0,\\
\varrho_k(0,x)=1.
\end{split}\label{bk-cont}
\end{equation}
Then $\varrho_k(t,x)=\partial_x X_{b_k}(t,X_{b_k}^{-1}(t,x))$. The lack of strong convergence of $\partial_xX_{b_k}$ and  \eqref{0} imply that although $\varrho_k\weak 1$ weakly in $L^1_{loc}$, it does not converge strongly. 
To prove \eqref{1} we can simply solve the ODE explicitly:
\begin{equation}
\begin{split}
\frac{\partial X_{b_k}(t,x)}{\partial  t}&=\frac{1}{k}\sin(kX_{b_k}(t,x)),\\
X(0,x)&=x.
\end{split}
\end{equation}
If $x\notin 2\pi{\mathbb Z}$, then for all $t\in[0,T]$ we have $\frac{\partial X_{b_k}(t,x)}{\partial t}\neq0$, thus for such $x$ we have
\begin{equation}
\int_0^t\frac{k\frac{\partial X_{b_k}(t,x)}{\partial t}}{\sin(kX_{b_k}(t,x))}\;ds=t.
\end{equation}
Changing the variables $\tau=k X_{b_k}(t,x)$ and integrating yields
\begin{equation}
\ln\tan\frac{k X_{b_k}(t,x)}{2}-\ln\tan\frac{kx}{2}=t
\end{equation}
which becomes
\begin{equation}\label{2}
\tan\frac{k X_{b_k}(t,x)}{2}=e^t \tan\frac{kx}{2}.
\end{equation}
Differentiating \eqref{2} in $x$ gives
\begin{equation}
\frac{\partial X_{b_k}(t,x)}{\partial x}(1+\tan^2 \frac{X_{b_k}(t,x)}{2})=e^t(1+\tan^2\frac{kx}{2}) 
\end{equation}
and using \eqref{2} 
\begin{equation}
\frac{\partial X_{b_k}(t,x)}{\partial x}=\frac{e^t(1+\tan^2\frac{kx}{2})}{1+e^{2t}\tan^2\frac{kx}{2}}
\end{equation}
which is
\begin{equation}\label{3}
\frac{\partial X_{b_k}(t,x)}{\partial x}=\frac{e^t}{\cos^2\frac{kx}{2}+e^{2t}\sin^2\frac{kx}{2}}.
\end{equation}
Observe that although \eqref{3} has been derived for all $x\notin 2\pi{\mathbb Z}$, it is easy to check that it is valid for all $x\in\R$. Then $\frac{\partial X_{b_k}(t,x)}{\partial x}=F(t, kx)$ where $x\mapsto F(t, x)=\frac{e^t}{\cos^2\frac{x}{2}+e^{2t}\sin^2\frac{x}{2}}$ is a periodic function with period $\pi$, from which we conclude that $F_k (t, \cdot)$ converges weakly, but not strongly, to the constant function
\[
\frac{1}{\pi} \int_0^\pi \frac{e^t}{\cos^2 \frac{x}{2} + e^{2t} \sin^2 \frac{x}{2}}\, dx\, .
\]
Notice that the latter integral can be explicitely computed using the substitution $y = e^t \tan \frac{x}{2}$ and becomes
\[
\frac{1}{\pi} \int_0^\infty \frac{2}{1 + y^2}\, dy = 1\, ,
\]
consistently with the (weak) convergence of the density $\rho_k$ to the constant $1$.

\bibliographystyle{abbrv}
\bibliography{hyper}

\end{document}